\documentclass[11pt]{amsart}

\usepackage{amssymb}
\usepackage{graphicx}
\usepackage{url}
\usepackage{todonotes}
\usepackage{enumitem}
\usepackage{hyperref}
\usepackage{amsthm}
\usepackage{thmtools}
\usepackage{thm-restate}

\newtheorem{thm}{Theorem}[section]

\newtheorem{lem}[thm]{Lemma}
\newtheorem{prop}[thm]{Proposition}

\newtheorem{question}[thm]{Question}
\newtheorem{conj}{Conjecture}
\newtheorem{rmk}[thm]{Remark}

\theoremstyle{definition}

\newcommand{\SL}{\mathrm{SL}(n,\mathbb{R})}
\newcommand{\SLZ}{\mathrm{SL}(n,\mathbb{Z})}
\newcommand{\PSL}{\mathrm{PSL}(n,\mathbb{R})}
\newcommand{\GL}{\mathrm{GL}(n-1,\mathbb{R})}
\newcommand{\Sone}{\mathbb{S}^1}
\newcommand{\Sn}{\mathbb{S}^n}
\newcommand{\Snone}{\mathbb{S}^{n-1}}
\newcommand{\SO}{\mathrm{SO}(n)}
\newcommand{\SOone}{\mathrm{SO}(n-1)}
\newcommand{\cO}{\mathcal{O}}
\newcommand{\cG}{\mathfrak{g}}
\newcommand{\SU}{\mathrm{SU}(n)}
\newcommand{\RP}{\mathbb{RP}^{n-1}}
\newcommand{\R}{\mathbb{R}}
\newcommand{\Z}{\mathbb{Z}}
\newcommand{\C}{\mathbb{C}}
\newcommand{\RO}{\R^n\backslash\{0\}}
\newcommand{\X}{\mathfrak{X}(\Sigma)}
\newcommand{\Gr}{\mathrm{Gr}(2,n)}

\begin{document}

\title{REAL ANALYTIC $\SL$-ACTIONS ON CLOSED MANIFOLDS}

\author{Miri Son}
\address{Department of Mathematics, Rice University, Houston, TX 77005}

\email{ms235@rice.edu}

\begin{abstract}
We classify real-analytic $\SL$-actions on closed manifolds of dimension $m$ for $3\leq n\leq m\leq2n-3$, which extends Fisher--Melnick's work for $\SL$-actions on closed $n$-manifolds. Additionally, we classify smooth $\SL$-actions on closed $m$-manifolds that are fixed-point free. As a corollary, we obtain the density or non-density of structural stability of fixed-point free $\SL$-actions.  
\end{abstract}

\maketitle

\section{Introduction}
In this paper, we classify real-analytic $\SL$-actions on closed manifolds of dimension $m$, where $3\leq n\leq m\leq2n-3$. The manifolds that arise in this classification are constructed as suspensions or suitable modifications based on all possible homogeneous spaces. We impose the upper bound $m\leq2n-3$ to constrain the variety of possibilities that arise when $m>2n-3$. A more detailed discussion of it is provided in Section \hyperref[s3.1]{3.1}, especially in Remark \hyperref[rmk3.3.]{3.3}. 

The $\SL$-actions considered in this classification fall into two distinct types, distinguished by the existence of global fixed points. A global fixed point means a fixed point that is fixed by the $\SL$-action on the manifold. 

Let $Q\leq\SL$ denote the stabilizer of a line, and $Q_2\leq\SL$ denote the stabilizer of a 2-plane in the standard representation of $\R^n$. Both $Q$ and $Q_2$ are maximal parabolic subgroups of $\SL$. In the first type of actions, those without global fixed points, the $\SL$-action is induced by the actions of either $Q$ or $Q_2$ on a certain submanifold $\Sigma$ of a manifold. These subgroup actions will factor through one-dimensional $\R^*\cong\R\backslash\{0\}$ actions by flows on $\Sigma$. 

\begin{restatable}{thm}{thmA} \label{thmA}
    Suppose that $G\cong\SL$ acts real-analytically on a closed manifold $M$ of dimension $m$, where $3\leq n\leq m\leq2n-3$. Suppose the action is nontrivial and admits no global fixed points. Then the $G$-action on $M$ is real-analytically and $G$-equivariantly isomorphic to one of the following:
    \begin{enumerate}
        \item The action on manifold is induced by an action of a maximal parabolic subgroup $Q$ on a submanifold $\Sigma$, as described in Section \hyperref[s4.1.1]{4.1.1}, so that 
        $$M=G\times_Q\Sigma$$ which is real-analyitically diffeomorphic to a fiber bundle over either $\Snone$ or $\RP$.
        \item The action on manifold is induced by an action of a maximal parabolic subgroup $Q_2$ on a circle $\Sigma_2$, as described in Section \hyperref[s4.1.2]{4.1.2}, so that $$M=G\times_{Q_2}\Sigma_2$$ which is real-analytically diffeomorphic to a circle bundle over $\mathrm{Gr}(2,n)$.
        \item The manifold $M$ is a homogeneous space with a transitive $\SL$-action, as classified in Theorem \hyperref[thm3.1.]{3.1}.
    \end{enumerate}
    The same classification applies to smooth actions as well.
\end{restatable}

Remark that the case (1) described in the preceding theorem accounts for all possible non-transitive $\SL$-actions on $m$-manifolds $M$ for $3<n\leq m<2n-3$. When the manifold has dimension $2n-3$, both cases (1) and (2) are possible.

\medskip

The second type of action occurs when the $\SL$-action admits fixed points on the manifold. In this case, the action on the manifold is constructed by inducing the action of the identity component $Q^0$ of $Q$ on a certain submanifold as in the previous case, and by attaching copies of $\R^m$ at the fixed points.  

\begin{restatable}{thm}{thmB} \label{thmB}
    Suppose that $G\cong\SL$ acts real-analytically on a closed $m$-manifold $M$, where $3\leq n\leq m\leq2n-3$. Assume that the $G$-action on $M$ is non-trivial and has at least one global fixed point. Then the manifold $M$ with $G$-action is, up to $G$-equivariant real-analytic isomorphism, constructed by gluing the following:
    \begin{enumerate}
        \item a suspension space $G\times_{Q^0}(\Sigma^0\backslash F)$, which is induced by an action of the identity component $Q^0$ of a maximal parabolic subgroup $Q$ on a certain submanifold $\Sigma^0$, except for the $G$-fixed points set $F\subset\Sigma^0$, and
        \item a neighborhood of $F$, which is $G$-equivariantly diffeomorphic to $\R^n\times F$, where $G$ acts in the standard way on $\R^n$ and trivially on $F$.
    \end{enumerate}
    A concrete construction is provided in Section \hyperref[s4.2]{4.2}.
\end{restatable}

In the cases where the real-analytic $\SL$-action admits no global fixed points, the real-analytic action can be replaced by a smooth one resulting in the same classification. However, when the $\SL$-action on the manifold admits global fixed points, the proof in this paper does not apply to the smooth actions due to the absence of a smooth linearization, see the discussion in Section \hyperref[s2.3]{2.3} and Question \hyperref[question2.5.]{2.5}.

\subsection*{Density of Structural Stability}
The main theorem concerning smooth $\SL$-actions on manifolds without global fixed points leads to further remarks on the density of structural stability. A vector field is said to $structurally$ $stable$ if small perturbations in the $C^r$-topology do not change the topological structure of its solution curves. Since the $\SL$-actions on manifolds are determined by a smooth vector field $X$ on $\Sigma$ or $\Sigma_2$, the density of structurally stable vector fields on $\Sigma$ or $\Sigma_2$ implies the density of structurally stable $\SL$-actions on manifolds. We prove the following; see Section \hyperref[s6]{6} for a detailed discussion.

\begin{restatable}{cor}{corA}\label{cor1.3.}
        Suppose that $G\cong\SL$ acts smoothly on a closed $m$-dimensional manifold $M$, where $3\leq n\leq m\leq2n-3$. Assume the $G$-action does not admit global fixed points. Then the following hold:
    \begin{enumerate}
        \item Case (1) with $m=n$ or Case (2) of Theorem \hyperref[thmA]{1.1}, the space of structurally stable $C^r$-diffeomorphic $G$-actions on manifolds $M$ is dense for $r\geq1$.
        \item Case (1) with $m=n+1$ of Theorem \hyperref[thmA]{1.1}, then the space of structurally stable $C^r$-diffeomorphic $G$-actions on manifold $M$ is dense for $r=1$ and not dense for $r>1$.
        \item Case (1) with $m>n+1$ of Theorem \hyperref[thmA]{1.1}, then the space of structurally stable $C^r$-diffeomorphic $G$-actions on manifolds $M$ is not dense for $r\geq1$.
        \item Case (3) of Theorem \hyperref[thmA]{1.1}, then the actions cannot be deformed.
    \end{enumerate}
\end{restatable}

\subsection{Previous work on simple Lie group actions on manifolds} We briefly review previous classification results of $\SL$-actions on manifolds. 

In low-dimensional cases, Schneider provided a classification of real analytic actions of $\mathrm{SL}(2,\mathbb{R})$ on closed surfaces and $\R^2$ in \cite{S}. 

Uchida classified real-analytic actions of $\SL$ on the standard $n$-sphere $\Sn$ for $n\geq 3$ in \cite{U1}. This was later extended to the standard $m$-sphere $\mathbb{S}^m$ for $5\leq n\leq m<2n-1$ in \cite{U2}. A key aspect of these works is the use of local linearizability, which examines whether actions of semisimple Lie groups near fixed points are locally conjugate to their linear actions on the tangent space. The linearization theorem for real-analytic actions of semisimple Lie groups was established by Guillenmin--Sternberg \cite{GS} and Kushnirenko \cite{K}.

Building on Uchida's work, Fisher--Melnick achieved a complete classification of both real-analytic and smooth actions of $\SL$ on closed $n$-dimensional manifolds, as presented in \cite{FM}. Their extension to smooth actions was made possible by the result of Cairns--Ghys in \cite{CG}, which established smooth linearizability of $\SL$-action. Moreover, Cairns and Ghys classified all orbit types of $\SL$-actions on $n$-dimensional manifolds in the same paper, which made it possible for Fisher and Melnick to generalize the classification from $n$-spheres to arbitrary closed $n$-dimensional manifolds.

In this paper, we extend these classification results to real-analytic actions of $\SL$ on higher-dimensional manifolds, those of dimension greater than $n$. Two key ingredients in our approach are the linearization theorem for real-analytic $\SL$-actions in Section \hyperref[s2.3]{2.3} and a new orbit classification, presented in Section \hyperref[s3.1]{3.1}.

\subsection{Rigidity and Zimmer program}
The main motivation of this work comes from the Zimmer program: Actions of higher-rank semisimple Lie groups and their lattices on closed manifolds arise from algebraic constructions. If the manifold dimension is smaller than the rank of the group, the action is very restricted. This is the famous Zimmer's conjecture proved by Brown--Fisher--Hurtado. They showed that actions of lattices in $\SL$ on manifolds of dimension at most $n-2$ and volume-preserving actions on manifolds of dimension $n-1$ should factor through finite group actions, in \cite{BFH1, BFH2, BFH3}. Their result was extended to several other semisimple Lie groups, each with corresponding dimension bounds in \cite{BFH3}. 

Brown--Rodriguez-Hertz--Wang completed $\SL$-action classification without the volume-preserving assumption through dimension $n-1$. An action of a lattice in $\SL$ on an $(n-1)$-dimensional manifold either factors through a finite group or extends to an $\SL$-action up to finite cover in \cite{BRW}. Note that there are exactly two faithful actions of $\SL$ on $(n-1)$-manifolds; the actions on $\RP$ and the lift of that action to $\Snone$. For additional context on the Zimmer program, see, for example, the survey \cite{F}.

For manifolds of dimension $n$, even actions of $\SL$ are complicated with several examples illustrating their complexity. Fisher--Melnick \cite{FM} provided a complete classification of both real analytic and smooth $\SL$-actions on closed $n$-dimensional manifolds. For lattice actions, several exotic examples exist. One well-known is the construction by Katok--Lewis \cite{KL}, in which $\SLZ$-actions that are not conjugate to algebraic actions, are constructed by blowing up the fixed point at the origin on the torus $\mathbf{T}^n$. Fisher--Melnick also introduced another exotic example in their paper. However, a complete classification of lattice actions of $\SL$ remains open.
\begin{conj} [{\cite[Conjecture 3.6]{FM}}]
    Let $\Gamma<\SL$ be a lattice and $M$ a compact manifold of dimension $n$. For any smooth action $\rho:\Gamma\rightarrow$ Diff($M$), one of the following holds:
    \begin{enumerate}
        \item $\rho$ extends to an action of $\SL$ or $\widetilde{\SL}$,
        \item $\rho$ factors through a finite quotient of $\Gamma$, or
        \item $\rho$ is built from tori, $G$-tubes, $G$-disks, blow-ups, and two-sided blow-ups, with $\Gamma$ a finite-index subgroup of $\SLZ$.
    \end{enumerate}
\end{conj}

For further details and relevant definitions, we refer to Fisher--Melnick \cite{FM}. One of the next natural questions concerns the case where the dimension of manifolds is greater than $n$. The main result of this paper lays the groundwork for the Zimmer program, which explores lattice actions of $\SL$ on higher-dimensional manifolds, as $\SL$-actions on manifolds classified in this paper naturally restrict to actions of its lattice. 

\subsection{Outline of the paper}
In Section \hyperref[s2]{2}, we present the necessary preliminaries for this paper. Specifically, we introduce the orbit classification result in Section \hyperref[s2.1]{2.1}, define the suspension space in Section \hyperref[s2.2]{2.2}, and review linearization results in Section \hyperref[s2.3]{2.3}.
In Section \hyperref[s3]{3}, we prove propositions concerning orbit types and global fixed points. 
In Section \hyperref[s4]{4}, we construct two distinct types of actions. 
In Section \hyperref[s5]{5}, we prove the main results.
To conclude, we comment on the density of structural stability in Section \hyperref[s6]{6}.

\subsection*{Acknowledgement}
The author thanks Michael Larsen for helpful discussion on reductive subgroups, and Insung Park for helpful discussion and comments, especially on the density of structural stability. She also thanks Homin Lee and Ralf Spatzier for reading an early draft and providing valuable feedback. Finally, she expresses her deep gratitude to her Ph.D. advisor, David Fisher, for numerous discussions and for supporting the work to completion.

\section{Preliminaries} \label{s2}
\subsection{Classification of orbit types} \label{s2.1}
The orbit classification of $\SL$-actions plays a key role in the results of the main theorems. Cairns--Ghys \cite[Theorem 3.5]{CG} proved the orbit classification for manifolds of dimension at most $n$. A detailed version of this result, along with its proof, also appeared in Fisher--Melnick \cite[Theorem 2.1]{FM}. 

\begin{thm} [{\cite[Theorem 3.5]{CG},\cite[Theorem 2.1]{FM}}]
    \label{thm2.1.}
    Let $G$ be connected and locally isomorphic to $\SL$ with $n\geq 3$, and assume $G$ acts continuously on a topological manifold M of dimension $n$. For any $x\in M$, the orbit $G.x$ is equivariantly homeomorphic to one of the following:
    \begin{enumerate}
        \item a point;
        \item $\Snone$ or $\RP$ with the projective action;
        \item $\R^n\backslash\{0\}$ with the restricted linear action or $(\R^n\backslash\{0\})/\Lambda$, for $\Lambda$ a discrete subgroup of the group of scalars $\R^*$;
        \item one of the following closed, exceptional orbits, or a finite cover:
         \begin{itemize}
             \item For $n=3$, $\mathcal{F}^3_{1,2}$, the variety of complete flags in $\R^3$
             \item For $n=4$, $\mathrm{Gr}(2,4)=\mathcal{F}^4_2$, the Grassmannian of $2$-planes in $\R^4$
         \end{itemize}
    \end{enumerate}
\end{thm} 
In Section \hyperref[s3.1]{3.1}, we will extend the orbit classification for $\SL$-actions on manifolds of dimension greater than $n$.

\subsection{Suspension space}\label{s2.2}
We introduce the concept of a suspension, also known as an induced module in representation theory. It is a useful construction of extending an action $\alpha$ of a subgroup $H\leq G$ on a space $M$ to a larger group $G$ acting on an associated bundle $M^\alpha$, as described below.

Let $H$ be a closed subgroup of a Lie group $G$, and assume that $H$ acts continuously on a topological space $M$. Consider the space $G\times M$, where we define two actions: the right $H$-action and the left $G$-action $$(g,x)\cdot h=(gh, \alpha(h^{-1})(x)), ~ g'\cdot(g,x)=(g'g,x),$$ where $g,g'\in G, x\in M$ and $h\in H$. 

We define the quotient space $M^\alpha:=(G\times M)/H$, which has a fiber bundle structure over $G/H$ with fibers homeomorphic to $M$. For any $(g,x)\in G\times M$, we denote the equivalence class in $M^\alpha$ by $[g,x]$. Since the $G$-action on $G\times M$ commutes with the $H$-action, this gives rise to an induced left $G$-action on $M^\alpha$ by left translation on the first factor. Note that the $G$-action preserves the fibers of $M^\alpha$.

\subsection{Linearization results}\label{s2.3}
In this subsection, we recall the previous results for linearization. When a group $G$ acts on a manifold $M$ and fixes a point $x\in M$, the differential of the action induces a linear action on the tangent space $T_xM$ at $x$. If the action of $G$ on $M$ is locally equivalent to its linear action on $T_xM$, the action is called $linearizable$. 

For the compact group actions, we have the linearization result of Bochner--Cartan. Here, we consider group actions on $(\R^m,0)$ where the origin $0\in\R^m$ is fixed by the action.

\begin{thm} [{\cite{B}}]
\label{thm2.2.}
    For all $k=1,2,...,\infty$, every $C^k$-action of a compact group on $(\R^m,0)$ is $C^k$-linearizable.
\end{thm}

For non-compact group actions, Guillemin--Sternberg \cite{GS} and Kushinirenko \cite{K} proved that a real-analytic action of a semisimple Lie group $G$ is linearizable near any fixed point. Using this theorem, Uchida classified all analytic $\SL$-actions on $S^n$ for $n\geq 3$ in \cite{U1}, and on $S^m$ for $5\leq n\leq m\leq 2n-2$ in \cite{U2}. 

\begin{thm} [{\cite{GS,K}}]
\label{thm2.3.}
    Every analytic action of $\SL$ on $(\R^m,0)$ is analytically linearizable.
\end{thm}

Later, Cairns--Ghys established a smooth linearization result for $\SL$-actions on $\R^n$ fixing 0 in \cite{CG}. This result enabled Fisher--Melnick \cite{FM} to classify not only real analytic but also smooth $\SL$-actions on closed manifolds $M$ of dimension $n$, for $n\geq3$. 

\begin{thm} [{\cite{CG}}]
    For all $n>1$ and for all $k=1,...,\infty$, every $C^k$-action of $\SL$ on $(\R^n,0)$ is $C^k$-linearizable.
\end{thm}

In the absence of a smooth linearization result on $\R^m$ for $m>n$, we are unable to extend our classification to smooth actions on manifolds with fixed points. We leave this as an open question.

\begin{question} \label{question2.5.}
    For all $3\leq n\leq m<2n-2$, and $k=1,2,...,\infty$, every $C^k$-action of $\SL$ on $(\R^m, 0)$ is $C^k$-linearizable.  
\end{question}

Remark that Cairns--Ghys \cite[Section 9]{CG} provides an example of a non-linearizable smooth action of $\mathrm{SL}(3,\R)$ on $\R^8$. This counter-example is obtained by deforming the adjoint action of $\mathrm{SL}(3,\R)$ on its Lie algebra $\mathfrak{sl}(3,\R)\cong\R^8$. Their construction may be extended to produce examples of non-linearizable smooth $\SL$-action on $\R^{n^2-1}$ for $n\geq3$, demonstrating that $n^2-1$ is a strict upper bound for the dimension $m$ in the question. We tentatively establish an upper bound of $m$ as $2n-2$, based on the orbit classification result in the next section.

\section{The orbit types and Fixed points} \label{s3}
Throughout this section, we fix $G=\SL$. 
We investigate the orbit types and global fixed points of $G$-actions, both of which play a key role in constructing actions and proving the main results. 

\subsection{The extended orbit classification}\label{s3.1} 
We extend the orbit classification by considering all possible homogeneous spaces arising from $\SL$-actions on $m$-dimensional manifolds, where $m\leq2n-3$. No new orbit types appear until the dimension reaches $2n-4$. When $2n-4$, a new orbit $\mathrm{Gr}(2,n)$, the Grassmannian of $2$-planes in $\R^n$, arises. When $m\geq2n-2$, fundamentally different types of orbits come up, see Remark \hyperref[rmk3.3.]{3.3} for details. Let $Q\leq\SL$ denote the stabilizer of a line in the standard representation of $\R^n$, and $Q_2\leq\SL$ denote the stabilizer of a 2-plane in the standard representation of $\R^n$.

\begin{thm}\label{thm3.1.}
    Let $G\cong\SL$ act continuously on a topological $m$-dimensional manifold $M$, with $m\leq2n-3$. For any $x\in M$, the orbit $G.x$ is equivariantly homeomorphic to one of the following: 
    \begin{enumerate}
        \item a point;
        \item $\Snone$ or $\RP$ with the projective action;
        \item $\R^n\backslash\{0\}$ with the restricted linear action or $(\R^n\backslash\{0\})/\Lambda$, for $\Lambda$ a discrete subgroup of the group of scalars $\R^*$;
        \item $\mathrm{Gr}(2,n)$, the real Grassmannian of $2$-planes in $\R^n$;
        \item a suspension space $G\times_{Q_2}\R$ or $(G\times_{Q_2}\R)/\Lambda$, where $Q_2$ acts on $\R$ via a homeomorphism factoring through $\R^*$ and $\Lambda$ a discrete subgroup of the group of scalars $\R^*$;
        \item $\mathcal{F}^n_{1,2}$, the variety of partial flags $V_1\subset V_2\subset \R^n$, with $\mathrm{dim}(V_1)=1$ and $\mathrm{dim}(V_2)=2$;
        \item finitely many $\mathrm{SL}(4,\R)$-action on $5$-dimensional homogeneous spaces arising from reductive subgroups, see Remark \hyperref[rmk3.4.]{3.4}.
    \end{enumerate}
\end{thm}
    
\begin{proof}
    Suppose that a homogeneous space $\cO_x = G.x$ is an orbit of $G$-action on a manifold, which is identified with $G/G_x$ where $G_x \leq G$ is a closed subgroup. 
    
    If the orbit has dimension less than or equal to $n$, then it is one of the cases (1), (2), or (3) by Theorem \hyperref[thm2.1.]{2.1}.

    Let $\cO_x$ be a $k$-dimensional $G$-orbit with $n<k\leq m$. Consider the complexified representation of the Lie algebra $\cG_x\otimes\C \subset 
    \cG_\C = \mathfrak{sl}(n,\C)$, where $\cG_x$ denotes the Lie algebra of $G_x$. Suppose that $\cG_x\otimes\C$ is a reductive subalgebra of $\cG_\C$ with the codimension $k$. By taking a suitable real form of $\cG_x\otimes\C$, we obtain a compact subgroup $K$ in $\SU$ whose real codimension is $k$. However, there is no compact subgroup $K$ of $\SU$ with codimension $k$ satisfying $4<n<k<2n-2$. This is because the highest-dimensional proper compact subgroup of $\SU$ is $\mathrm{U}(n-1)$, which has codimension $2n-2$, when $n>4$, see below Lemma \hyperref[lem3.2.]{3.2}. When $n=4$, there exist homogeneous spaces of dimension $5$, which correspond to case (7). 
    
    Now assume that $\cG_x\otimes\C$ is not a reductive subalgebra of $\cG_\C$, which implies that the isotropy representation on $V=(\cG/\cG_x)_\C\cong\C^k$ is reducible. There exists a nontrivial invariant complex subspace $W$ of dimension $p$, for some $0<p<k$. The stabilizer of such a $p$-dimensional subspace has codimension $p(n-p)$, so that $0<p(n-p) \leq k$. Under the assumptions $4\leq n<k\leq2n-3$, the only possible values for $p$ are $p=1$, $p=2$, $p=n-2$ or $p=n-1$. For $p=n-1$ or $p=n-2$, replacing $G_x$ by the outer automorphism $g\mapsto(g^{-1})^t$ reduces the problem to the case $p=1$ or $p=2$. 

    Consider $p=1$ with the invariant complex line $W$. If $W\cap\overline{W}=0$, $\mathfrak{g}_x\otimes\C$ preserves the flag $W\subset U=W\oplus\overline{W}\subset\C^n$. The stabilizer of such a flag has codimension $2n-3$, and assume that it equals $\mathfrak{g}_x\otimes\C$ based on our dimension assumptions. The intersection $U_0=U\cap\overline{U}$ is a real $2$-dimensional and $\mathfrak{g}_x$-invariant. The $\mathfrak{g}_x\otimes\C$-invariant decomposition $U=W\oplus\overline{W}$ defines an $\mathfrak{g}_x$-invariant complex structure on $U_0$ as the following: for any $u\in U$, there is a unique $w\in W$ such that $$u=\frac{1}{2}(w+\overline{w}).$$
    Then, $$J(u):=\frac{i}{2}(\overline{w}-w)$$ defines $\mathfrak{g}_x$-equivariant automorphism of $U_0$ with $J^2=-\mathrm{Id}$. Then $\mathfrak{g}_x$ is contained in the stabilizer of a $2$-plane $U_0$ in $\R^n$ together with a complex structure on $U_0$. And the codimension of this stabilizer in $\mathfrak{sl}(n,\R)$ is $2n-2$, which is excluded by our dimension assumptions.

    Now assume that $W^0=W\cap\overline{W}$ is a real line, $\mathfrak{g}_x$ has codimension $1$ in its stabilizer. If $\mathfrak{g}_x$ acts reducibly on $V_0/W_0$, then it is conjugate to a parabolic subalgebra consisting of upper triangular matrices with block sizes $1,1,n-2$, and trace zero. In this case, the homogeneous space $\mathcal{F}^n_{1,2}$ having dimension $2n-3$ appears, which corresponds to case (6).
    
    Let $\cG_x$ act irreducibly on $V_0/W_0$. Consider the parabolic subalgebra $\mathfrak{q}$ stabilizing a line in $\R^m$, and $\cG_x\subset \mathfrak{q}$ has codimension at least 2 in $\mathfrak{q}$. The parabolic subgroup $Q$ has Levi decomposition $Q\cong\GL\ltimes\R^{n-1}$, where the unipotent radical $\R^{n-1}$ acts reducibly on $V_0/W_0$. Since $G_x$ acts irreducibly, $G_x$ must lie entirely within the Levi subgroup $\GL$. Consequently, $G_x$ projects onto a proper, closed, irreducible subgroup of $\GL$ whose codimension lies between $2$ (when $k=n+1$) and $n-4$ (when $k=2n-5$). However, no such subgroup exists for $n\geq 4$, leading to a contradiction. 
    
    Now consider the case when $p=2$ and $n\geq4$ with the invariant $2$-dimensional complex subspace $T\subset V$. If $T\cap\overline{T}=0$, then $\cG_x\otimes\C$ preserves the flag $T\subsetneq T\oplus\overline{T}\subset\C^{2n-4}$. The stabilizer of such a flag has codimension $2(n-2)+2(n-4)$ which is strictly greater than $2n-3$ when $n>3$. Therefore, this case cannot occur.

    Now assume that $T_0=T\cap\overline{T}$ is a real line, and $\cG_x$ has codimension one in its stabilizer. The reducible cases for $\cG_x$ have already been addressed in the above cases for $p=1$, so here we focus on the case where $\cG_x$ acts irreducibly on $V_0/T_0$ In this setting, $\cG_x$ is a codimension one subalgebra of the stabilizer $\mathfrak{q}_2$ of a $2$-dimensional complex subspace $T\subset V$ in $\R^n$. The corresponding parabolic subgroup $Q_2$ admits Levi decomposition $$Q_2\cong(\mathrm{GL}(2,\R)\times\mathrm{GL}(n-2,\R))\ltimes\R^{2n-4},$$ where the unipotent radical $\R^{2n-4}$ acts reducibly on $V_0/T_0$. Because $G_x$ acts irreducibly, $G_x$ must be contained entirely within the Levi part $\mathrm{GL}(2,\R)\times\mathrm{GL}(n-2,\R)$ as a closed, codimension one subgroup that remains irreducible on $V_0/T_0$. The only possibility for such a subgroup is that $\mathfrak{g}_x$ projects onto $$\mathfrak{sl}(2,\R)\times\mathfrak{sl}(n-2,\R)\subset\mathfrak{gl}(2,\R)\times\mathfrak{gl}(n-2,\R).$$ Let $E_2^0$ denote the connected, normal subgroup of $Q_2$ isomorphic to $(\mathrm{SL}(2,\R)\times\mathrm{SL}(n-2,\R))\ltimes\R^{2n-4}$, and $E_{\Lambda,2}\lhd Q_2$ the preimage of $\Lambda<Q_2/E_2^0\cong\R^*$. The possibilities described in case (5) correspond to $G_x=E^0_2$ or $E_{\Lambda,2}$.
    
    Finally suppose that  $\mathrm{dim}_\R(T\cap\overline{T})=2$. In this case, $\cG_x\otimes\C$ does not preserve any proper subgroup of $T$. $\cG_x$ is contained in the stabilizer of a $2$-dimensional subspace of $\R^n$. This stabilizer has real codimension $2(n-2)$ which equals the codimension of the orbit $G/G_x$. Therefore, $G_x=\mathrm{stab}_G(\R^2)$ and the homogeneous space $\mathrm{Gr}(2,n)$ appear, which correspond to case (4).
\end{proof}

The following lemma represents the smallest possible dimension for orbits arising from a reductive subgroup. 

\begin{lem}[{\cite{Q}}]\label{lem3.2.}  
    The highest-dimensional proper compact subgroup of $\SU$ is $\mathrm{U}(n-1)$ when $n>4$.
\end{lem}
\begin{proof}
    Consider elements of $\SU$ as $n\times n$ unitary matrices. Recall that $\SU$ is a reductive group. Let $H$ be a reductive proper subgroup of maximal possible dimension with a faithful representation $V$ of dimension $n$. We claim that, up to conjugacy,  $H$ is $\mathrm{U}(n-1)$ which has dimension $(n-1)^2+1$.

    Suppose first that $V$ decomposes as a direct sum of two subrepresentations of dimension $a$ and $b$, with $a+b=n$. Then the dimension of $H$ is at most $a^2+b^2$ which is less than or equal to $(n-1)^2+1$. This implies that $V$ must be irreducible, except possibly in the case $a=1$.

    Next, suppose that $V$ is isomorphic to a tensor product $V=V_1\otimes V_2$, where $\mathrm{dim} V_1=a$, $\mathrm{dim}V_2=b$, and $ab=n$. Then the dimension of $H$ is at most $a^2+b^2$, and the maximal value of $a^2+b^2$ is $(n/2)^2+(n/2)^2=n^2/2$, which is strictly less than $(n-1)^2+1$ for all $n>4$. Hence, $V$ cannot be a form of tensor product of irreducible representations.
    
    Since $V$ is an irreducible representation that is not a tensor product of lower-dimensional irreducible representations, $H$ is a simple compact Lie group with the minimal dimension of faithful representation $V$ which is well-known. Thus, $H=U(n-1)$ of dimension $(n-1)^2+1$.
\end{proof}

\begin{rmk}\label{rmk3.3.}
    The above lemma shows that new orbit types appear, such as $\SU/\mathrm{U}(n-1)\cong\mathbb{CP}^{n-1}$ of dimension $2n-2$, and $\SU/\mathrm{SU}(n-1)\cong\mathbb{S}^{2n-1}$ of dimension $2n-1$. Additionally, $\SL/\GL$ is another homogeneous space of dimension $2n-2$. Therefore, the dimension bound $2n-2$ marks a natural stopping point for this paper, as several new orbit types begin to appear. 
\end{rmk}

\begin{rmk}\label{rmk3.4.}
    For $n=4$, there exists a compact orbit of dimension $5$ arising from the compact symplectic subgroup $\mathrm{Sp}(2)$ in $\mathrm{SU}(4)$. In what follows, we rule out the compact orbit for the $\mathrm{SL}(4,\R)$-action on a manifold $M$ of dimension $5$.
\end{rmk}

\subsection{The Fixed Points}\label{s3.2}

Consider a $G$-action on a closed manifold $M$. Let $K$ be the maximal compact subgroup of $G$. There exists a $K$-invariant Riemannian metric $\kappa$ on the closed manifold $M$, which is obtained by averaging any Riemannian metric on $M$ over $K$ with respect to the Haar measure.

By applying the linearization theorems for compact group action in Theorem \hyperref[thm2.2.]{2.2} and using the orbit classification in Section \hyperref[s3.1]{3.1}, we obtain information about the fixed points of $G$-action.

\begin{prop} \label{prop3.5.}
    For a nontrivial real-analytic or smooth action of $G\cong\SL$ on a connected $m$-dimensional manifold $M$ for $3\leq n\leq m\leq2n-3$, the fixed points set of $G$-action is a finite union of closed $(m-n)$-dimensional submanifolds.
\end{prop}
\begin{proof}
    Let $x\in M$ be a $G$-fixed point. If the isotropy representation of $G$ at $x$ is trivial, then actions of the maximal compact subgroup $K$ are trivial in a neighborhood of $x$ via the exponential map of $\kappa$. Then the $K$-action is trivial on all of $M$, and so is $G$. 
    
    Thus, the isotropy representation of $G$ at $x$ is nontrivial. By Theorem \hyperref[thm2.2.]{2.2}, any $K$-action on a neighborhood at $K$-fixed points is linearizable via the exponential map with respect to the metric $\kappa$. This representation can be identified with the standard representation of $K$, which is one of the canonical representations in $\R^m$. Under this representation, the set of $K$-fixed points has dimension $(m-n)$ where $3\leq n\leq m\leq 2n-1$. 
    
    We now claim that the set of $K$-fixed points coincides with the set of $G$-fixed points. Suppose not, then there exists a $K$-fixed point $y$ that is not fixed under the $G$-action. In this case, the stabilizer $G_y$ of $y$ contains $K$. However, by the orbit classification theorems in Section \hyperref[s3.1]{3.1}, where $3\leq n\leq m\leq2n-3$, the point stabilizer is conjugate to parabolic subgroups or its subgroups, none of which can contain a subgroup isomorphic to $K$. This contradiction establishes the claim.
    
    Since $K$ acts by isometries, the fixed-point set forms closed, totally geodesic submanifolds with respect to $\kappa$. Moreover, since $M$ is compact, the fixed-point set consists of a finite union of connected components. 
\end{proof}

\begin{rmk}\label{rmk3.6.}
    There are two nontrivial irreducible $\SL$-representations on $\R^n$, up to conjugation: the standard representation $\rho$, and the dual representation, defined by $\rho^*(g)=\rho((g^{-1})^t)$. Under $\rho$, there is a $Q$-invariant line, whereas there is no $Q$-invariant line under $\rho^*$ for $n\geq3$. For $3\leq n<m\leq 2n-1$, the only nontrivial $\SL$-representations on $\R^m$ are $\rho\oplus\mathrm{id}$ or $\rho^*\oplus\mathrm{id}$, up to conjugation, where $\mathrm{id}$ denotes the trivial representation. 
\end{rmk}

\section{Constructions}\label{s4}
In this section, we construct two families of real-analytic nontrivial actions of $G=\SL$ on closed manifolds $M$ of dimension $m$, for $3\leq n\leq m\leq2n-3$. In Section \hyperref[s5]{5}, we will prove that every nontrivial and non-transitive real-analytic action of $G$ on closed manifolds $M$ is conjugate to one of these constructions. When $n=m$, our construction is equivalent to the one by Fisher--Melnick \cite{FM}.

Recall that a maximal parabolic subgroup $Q<G$, defined as the stabilizer of a line in the standard representation of $\R^n$, is isomorphic to $\GL\ltimes\R^{n-1}$. 
Let $\{a^t\}$ denote the one-parameter subgroup generating the identity component of the center of a Levi subgroup $L\cong\GL$ of $Q$. Let $C\cong\mathrm{O}(n-1)$ be the maximal compact subgroup of $L$. Denote by $\pi$ the projection map from $Q$ to $L$, and let
$\sigma\in C$ be an involution projecting to $-1$ in $\R^*\cong L/[L,L]$. Define the following homomorphisms: 
$$\nu_0:Q^0\rightarrow\R, ~ q\mapsto \log(\det(\pi(q))),$$
$$\nu:Q\rightarrow \R^*\cong\Z_2\times\R, ~ q\mapsto (\mathrm{sgn}(\det(\pi(q))), \log|\det(\pi(q))|)$$
where $Q^0$ is the identity component of $Q$. Let a subgroup $E$ of $Q$ be the kernel of $\nu$. Then the kernel of $\nu_0$ coincides with the identity component $E^0$ of $E$. Note that $E^0$ is a connected normal subgroup of $Q$ isomorphic to $\mathrm{SL}(n-1,\mathbb{R)} \ltimes\R^{n-1}$.

\medskip

Let $Q_2<G$ be a maximal parabolic subgroup defined as the stabilizer of a plane in the standard representation of $\R^n$. By the Levi decomposition, it is isomorphic to $$(A\times B)\ltimes\R^{2n-2}:=(\mathrm{GL}(2,\R)\times\mathrm{GL}(n-2,\R))\ltimes\R^{2n-2}$$ where the Levi subgroup $L_2=\mathrm{GL}(2,\R)\times\mathrm{GL}(n-2,\R)$ has determinant one.
Let $\{a^t_2\}$ denote the one-parameter subgroup generating the identity component of the center of the Levi subgroup. Let $C_2\cong\mathrm{O}(2)\times\mathrm{O}(n-2)$ be a maximal compact subgroup of $L_2$.

There are exactly two codimension-one subgroups of $Q_2$: One is $(\mathrm{SL}(2,\R)\times\mathrm{SL}(n-2,\R))\ltimes\R^{2n-2}$, and the other is a parabolic subgroup with diagonal blocks of sizes $1\times1$, $1\times1$, and $(n-2)\times(n-2)$. Let $\pi_2$ denote the projection map that sends an element $q\in Q_2$ to its upper left $2\times 2$ minor, so that $\pi_2(q)\in A\cong\mathrm{GL}(2,\R)$. Let
$\sigma_2\in C_2$ be an involution element that projects to $-1$ in $\R^*$ which corresponds to the center of the Levi subgroup $L_2$. We now define a homomorphism:
$$\nu_2:Q_2\rightarrow\R^*\cong\Z_2\times\R, q\mapsto(\mathrm{sgn}(\mathrm{det}(\pi_2(q))), \mathrm{ln}|\mathrm{det}(\pi_2(q))|).$$
Note that the kernel of $\nu_2$ defines a codimension-one subgroup of $Q_2$ which is diffeomorphic to $(\mathrm{SL}(2,\R)\times\mathrm{SL}(n-2,\R))\ltimes\R^{2n-2}$.

\subsection{Construction 1: Without global fixed points}\label{s4.1} We present two examples of $G$-actions on closed manifolds $M$ without fixed points. The second example, Construction 1-(2), applies only to manifolds of dimension $m=2n-3$.

\subsubsection{Construction 1-(1)}\label{s4.1.1}
Let $\Sigma^0$ be a real analytic closed connected manifold of dimension $m-n+1$. Let $\{\psi^t_X\}$ be a real analytic flow on $\Sigma^0$, generated by a vector field $X$. Let $\tau$ be a real analytic involution either on $\Sigma^0$ commuting with $X$ or on $\Sigma^0 \times\{-1,1\}$ exchanging the two components. Let $\Sigma =\Sigma^0$ in the first case and $\Sigma^0\times\{-1,1\}$ in the second. In the second case, extend $X$ to $\Sigma$ by pushing it forward via $\tau$ to the other component.

Consider an action of $\R^*$ on $\Sigma$ as follows: $$t\mapsto\psi^{\log|t|}_X\circ\tau^{(1-\mathrm{sgn}(t))/2}.$$ Composing with the map $\nu : Q \rightarrow \R^*$, we obtain an induced action of $Q$ on $\Sigma$, which we denote by $\mu_{X,\tau}$. 
Consequently, the suspension space of the action 
$\mu_{X,\tau}$ is given by $$M=G\times_Q\Sigma$$ which is a closed $m$-dimensional manifold equipped with a real-analytic action of $G$.

\subsubsection{Construction 1-(2)}\label{s4.1.2}
Let $\Sigma^0_2$ be a real-analytic circle, and let $\{\psi^t_{2,X}\}$ be a real-analytic flow on $\Sigma^0_2$ generated by a vector field $X$. Let $\tau_2$ be a real-analytic involution that either acts on $\Sigma^0_2$ commuting with $X$, or acts on $\Sigma^0_2\times\{-1,1\}$ exchanging the two components. Let $\Sigma_2 =\Sigma^0_2$ in the first case and $\Sigma^0_2\times\{-1,1\}$ in the second. In the second case, extend $X$ to $\Sigma_2$ by pushing it forward to the second component via $\tau_2$.

Now define an action of $\R^*$ on $\Sigma_2$ by $$t\mapsto\psi^{\log|t|}_{2,X}\circ\tau_2^{(1-\mathrm{sgn}(t))/2}.$$ Composing this with a homomorphism $\nu_2 : Q_2 \rightarrow \R^*$, we obtain an induced action of $Q_2$ on $\Sigma_2$, denoted by $\mu_{2,X,\tau}$. The corresponding suspension space is then given by $$M=G\times_{Q_2} \Sigma_2$$ which defines a closed manifold of dimension $m=2n-3$, equipped with a real-analytic action of $G$.

\subsection{Construction 2: With global fixed points} \label{s4.2}
We now construct actions of $G$ on $m$-dimensional closed manifolds that contain global $G$-fixed points. As established in the Proposition \hyperref[prop3.5.]{3.5}, the set of global fixed points is a finite union of ($m-n$)-dimensional closed submanifolds of $M$.

Let $\Sigma^0$ be a real analytic, connected, closed $(m-n+1)$-dimensional manifold. Let $\Sigma^0$ contain the global fixed points set which is a nonempty, finite union of closed $(m-n)$-dimensional submanifolds, denoted by $F\subset\Sigma^0$. Each connected component of $\Sigma^0\backslash F$ forms a manifold with boundaries. We denote by $\Sigma^0_\alpha$ for each connected component.

Define $U_{\alpha\beta}$ to be a collar neighborhood of the boundary of each closure $\overline{\Sigma^0_\alpha}$ identified up to diffeomorphism with $[0, 1) \times V_{\alpha\beta}$, where $V_{\alpha\beta}=U_{\alpha\beta}\cap F$. 

Consider a real analytic vector field $X$ on $\Sigma^0$ that vanishes on $F$, and remains nonvanishing on each collar neighborhood $U_{\alpha\beta}\cong[0,1)\times V_{\alpha\beta}$ away from $\{0\}\times V_{\alpha\beta}\subset F$. Moreover, in collar coordinates $(t,x)\in[0,1)\times V_{\alpha\beta}$ (identified with $U_{\alpha\beta}$ via a diffeomorphism $f$ satisfying $f(0,x)=x$ for all $x \in V_{\alpha\beta}$), the vector field $X$ satisfies $$D_tX(t,x)|_{t=0}=1$$ for all $x\in V_{\alpha\beta}$ where $D_t$ denotes the directional derivatives in the $[0,1)$ direction.

Define a $Q^0$-action on $\Sigma^0\backslash F$ by letting the one-parameter subgroup $\{a^t\}$ of $Q^0$ act via the flow $\{\psi_{X}^t\}$ and then pulling it back through the epimorphism $\nu_0 : Q^0 \rightarrow \R^*_{>0}$. This action yields the space $$M' = G \times_{Q^0} (\Sigma^0\backslash F)$$ obtained as the suspension of the $Q^0$-action on $\Sigma^0\backslash F$.

Now consider the $G$-action on $\R^n \times F$ where $G$ acts in the standard way on $\R^n$ and trivially on $F$. For each $x\in F$, let $l_0\subset \R^n\times \{x\} \cong \R^n$ be a $Q^0$-invariant ray from the origin that is pointwise fixed by $E^0$. Then the restriction of $\{a^t\}$ to $l_0$ corresponds to the flow $\{\psi_{X}^t\}$ on the image of $f$ at $[0, 1)\times \{x\}$ in the collar neighborhood $U_{\alpha\beta}$ containing $x$.

By identifying the equivalence described above and extending it $G$ equivariantly, we obtain a real-analytic gluing of $(\R^n, 0) \times F$ to $M'$ along the boundaries of $\Sigma^0\backslash F$, resulting in a closed manifold $M$ of dimension $m$. The group $G$ acts on $M$ real-analytically and faithfully with a nonempty set of global fixed points.

\section{Proof of the Main Theorems} \label{s5}
In this section, we prove the main Theorems \hyperref[thmA]{1.1} and \hyperref[thmB]{1.2} which assert that all nontrivial actions of $\SL$ on manifolds arise from the constructions given in the previous Section \hyperref[s4]{4}. Before turning to the proofs of these results, we first consider closed submanifolds that remain invariant under the action of specific compact subgroups.

\begin{prop}\label{prop5.1.}
    Consider nontrivial actions of $G=\SL$ on a closed m-dimensional manifold $M$, where $3\leq n\leq m\leq2n-3$. Assume that $M$ only contains the standard orbit types (1), (2), or (3) which are described in Theorem \hyperref[thm3.1.]{3.1}. Then the fixed point set $\Sigma$ of the action of $C^0\cong\SOone$ is a nonempty, finite union of closed submanifolds of dimension $m-n+1$. 
\end{prop}
\begin{proof}
    By the orbit classification in Theorem \hyperref[thm3.1.]{3.1}, every standard $G$-orbit contains points fixed by $C^0$, so the fixed-point set $\Sigma$ is nonempty. Let $\kappa$ be a $K$-invariant metric on the manifold $M$, where $K$ is the maximal compact subgroup of $G$ containing $C^0$. Since compact group actions preserve $\kappa$, $\Sigma$ forms a finite union of closed, totally geodesic submanifolds with respect to the Riemannian metric $\kappa$. It remains to show that each connected component of $\Sigma$ has dimension $m-n+1$.
    
    Let $x\in\Sigma$. Assume that $x$ is fixed under the $G$-action. Then the restricted $K$-action is linearizable in a neighborhood of $x$. The isotropy representation of $K$ at $x$ extends to a representation of $G$, corresponding to the standard representation of $G$ on $\R^m$. Under this representation, the fixed set of $C^0$ has dimension $m-n+1$.

    Now suppose that $x$ is not fixed by the $G$-action. By assumption, the orbit $\cO_x$ can only be of two types; $(n-1)$-dimensional or $n$-dimensional homogeneous space. 
    
    If $\cO_x$ is a $(n-1)$-dimensional homogeneous space which is conjugate to $\Snone$ or $\RP$, then $C^0$ acts irreducibly on $T_x\cO_x$ and trivially on its $\kappa$-orthogonal complement. In this case, $\Sigma$ corresponds locally to a $\kappa$-geodesic submanifold orthogonal to $\cO_x$ at $x$, which has dimension $m-n+1$. 
    
    If $\cO_x$ has dimension $n$, then the $G$-action on this orbit $\cO_x$ is conjugate to the linear action on $\RO$, as described by Cairns and Ghys in \cite[Theorem 4.1]{CG}. The fixed-point set of $C^0$ within $\cO_x$ has dimension 1. As in the previous case, $C^0$ acts trivially on the $\kappa$-orthogonal complement of $T_x\cO_x$. Therefore, the $C^0$-fixed set in a neighborhood of $x$ is also of dimension $m-n+1$. 
\end{proof}

When the manifold $M$ contains an orbit conjugate to $\Gr$, it is necessary to consider other types of closed submanifolds.

\begin{prop}\label{prop5.2.}
    Let $G\cong\SL$ act nontrivially on a closed manifold $M$ of dimension $2n-3$ that contains an orbit conjugate to $\Gr$. Then the fixed point set $\Sigma_2$ of the $C^0_2\cong\mathrm{SO}(2)\times\mathrm{SO}(n-2)$-action is a nonempty finite union of circles.
\end{prop}
\begin{proof}
    Since the $\Gr$ orbit contains a fixed point of the restricted $C^0_2$-action, the fixed set $\Sigma_2$ is nonempty. As the fixed point set of a compact group action, $\Sigma_2$ is a finite union of closed, totally geodesic submanifolds with respect to the Riemannian metric $\kappa$, as in the preceding proposition. The fact that $\Sigma_2$ has dimension $1$ follows from the observation that, for $x\in\Sigma_2$, the $C^0_2$-action is irreducible on $T_x\Gr$ and trivial on its $\kappa$-orthogonal complement. 
\end{proof}

\subsection{$G$-action on manifolds without global fixed points} \label{s5.1}
If the $G$-action on the manifold admits no global fixed points, the following theorem shows that the actions are realized as suspensions of actions by maximal parabolic subgroups on submanifolds. Moreover, when $m\neq2n-3$, only Case (1) in the following theorem can occur.  
\thmA* \label{thm1.1.}

\begin{proof}
    When the $G$-action is not transitive, the manifold $M$ either contains an orbit conjugate to $\Gr$ or it does not. Accordingly, the proof is completed together with the following Lemmas \hyperref[lem5.3.]{5.3} and \hyperref[lem5.4.]{5.4} which will be proved below. In fact, when the dimension $m<2n-4$, only the first case can occur, since $\Gr$ cannot appear as an orbit in this dimension range. 

    It also applies to smooth $G$-actions on $M$ rather than real-analytic ones. If we begin with a smooth $G$-action, then the vector fields $X$ and involutions $\tau$ in the proofs will also be smooth, and the resulting $Q$ or $Q_2$-actions on $G\times \Sigma$ or $G\times \Sigma_2$ will be smooth as well. Consequently, the map $\Phi$ or $\Phi_2$ defines a smooth diffeomorphism between $M$ and $G\times_Q\Sigma$ or $G\times_{Q_2}\Sigma_2$.
\end{proof}

\begin{lem} \label{lem5.3.}
    Under the same setting and assumptions in Theorem \hyperref[thm1.1.]{1.1}, and assuming there are no orbits conjugate to $\Gr$, only case (1) of Theorem \hyperref[thm1.1.]{1.1} can occur.
\end{lem}

\begin{proof} 
    Let $\Sigma$ be the set of fixed points of the restricted actions of $C^0$, as defined in Proposition \hyperref[prop5.1.]{5.1}. For each $x\in\Sigma$, the point stabilizer $G_x$ of $x$ contains $C^0$. Note that, by the orbit classification in Theorem \hyperref[thm3.1.]{3.1}, all point stabilizers are conjugate in $G$ into $Q$. Consequently, there exists $h \in G$ such that $h^{-1}G_xh \leq Q$. In particular, this implies that $h^{-1}C^0 h \leq Q$, and therefore, $$C^0 \leq \mathrm{Stab}_G(hQ).$$ Here, the stabilizer of the $G$-action acts by translation on the left-coset space $hQ$. Recall that the following homogeneous spaces are $K$-equivariantly diffeomorphic: $$G/Q \cong \mathbb{RP}^n \cong K/C,$$ where $K \cong \mathrm{SO}(n)$ and $C \cong \mathrm{O}(n-1).$ From this diffeomorphism, we obtain $$\mathrm{Stab}_G(hQ) \cap K =\mathrm{Stab}_K(h'C)$$ for some $h' \in K$. Since $h'\in N_K(C^0)=C$ where $N_K(C^0)$ is the normalizer of $C^0$ in $K$, it follows that $h' \in C$, and thus $h \in Q$. Therefore, we conclude that $G_x \leq Q$ for all $x \in \Sigma$. 
    
    By the orbit classification results again in Theorem \hyperref[thm3.1.]{3.1}, the point stabilizer $G_x$ for all $x \in \Sigma$ is one of the subgroups $E^0, E, Q^0$, or $Q$ which are introduced in Section \hyperref[s4]{4}. Since all these subgroups contain $C^0$, it follows that $Q.\Sigma=\Sigma$. Moreover, as the normal subgroup $E^0$ in $Q$ acts trivially on $\Sigma$, the $Q$-action on $\Sigma$ factors through the epimorphism $$\nu : Q \rightarrow \R^*$$ in Section \hyperref[s4]{4}. Let $\Sigma^0$ be a connected component of $\Sigma$. The $\R^*_{>0}$ preserves $\Sigma^0$; this action is an analytic flow $\{\psi^t_X\}$, which corresponds to the restriction of $\{a^t\}$. Let $\sigma$ be an analytic involution in $C$ that maps to $-1$ under $\nu$. This involution leaves $\Sigma$ invariant. Let $\tau = \sigma|_\Sigma$. Depending on $\tau$, we set $\Sigma = \Sigma^0$ or $\Sigma^0 \times \{-1, 1\}$. Let $\mu_{X, \tau}$ denote the corresponding $Q$-action on $\Sigma$. For this $Q$-action, we define the $G$-equivariant map $$\Phi : G \times_Q \Sigma \rightarrow M, ~~ [(g,x)] \mapsto g.x .$$ 
    
    Now we show $\Phi$ is a diffeomorphism. First, we prove that $\Phi$ is surjective by showing it is both open and closed. Then, we establish injectivity using the Iwasawa decomposition. The image of $\Phi$ is closed because the fiber product is compact. To show $\Phi$ is
    open, we show the differential map has full rank. If the orbit $\cO_x$ is $(n-1)$-dimensional, corresponding to $G_x = Q^0$ or $Q$, then $\cO_x$ is equivariantly diffeomorphic to $\RP$ or $\Snone$. Then $T_x\cO_x$ and $T_x\Sigma$ are transversal since the $C^0$-fixed set in $\cO_x$ is $0$-dimensional. Thus, the differential of $\Phi$ at $[(e, x)]$ is onto $T_xM$. 
    In the other case, the orbit $\cO_x$ has dimension $n$ and is equivariantly diffeomorphic to $\R^n\backslash\{0\}$. Once again, consider the decomposition of $G$-action on $\R^n\backslash\{0\}$ by the $G$-action on $(G \times \R^+_*)/Q^0$ as described by Cairns--Ghys, see \cite[Theorem 4.1]{CG}. This decomposition implies that the radial line defined by $\R^+_*$ can only be included in $\Sigma$. Thus, the intersection $T_x\cO_x \cap T_x\Sigma$ has rank one, and the differential of $\Phi$ at $[(e,x)]$ is also onto in $T_xM$. 
    By equivariance of $\Phi$, the differential at $[(g, x)]$ is onto in $T_{g.x}M$ for all $g \in G$. We conclude that $\Phi$ is open, so it is a surjective local diffeomorphism. 

    Recall that the Iwasawa Decomposition is a diffeomorphism $$K \times A \times N \rightarrow \SL$$ where $K \cong \SO$, $A$ is the identity component of the diagonal subgroup, and $N$ is the group of unipotent upper-triangular matrices. As $N < E^0$ and $A/(A \cap E^0) \cong \{a^t\}$, the Iwasawa Decomposition provides a normal form for elements of $G \times_Q \Sigma$. For any $g = ka'n \in G$ and $x \in \Sigma$, we have $$[(g,x)]=[(k,a^t.x)]=[(k\sigma,\tau a^t.x)],$$ where $a'=a^ta''$ with $a''\in A\cap E^0$. This means that every equivalence class $[(g, x)] \in G \times_Q \Sigma$ has a representative of the form $(k, x)$ for some $k \in K$ and $x \in \Sigma^0$.

    Now suppose that $\Phi([(k, x)]) = \Phi([(k', x')])$ with $k, k' \in K$ and $x, x' \in \Sigma^0$. Then $k' = kq$ and $x' = q^{-1}.x$ for $q = k^{-1}k'$. By $G/Q\cong K/C$, $q$ is contained in $C$ which is the normalizer of $C^0$. Thus $[(k, x)] = [(k', x')]$, confirming that $\Phi$ is injective. Therefore, $\Phi$ is a real analytic diffeomorphism, and this completes the proof.
\end{proof}

Let us examine the possible diffeomorphism types of the resulting manifolds for the previous theorem. If the involution $\tau$ has no fixed points on $\Sigma^0$, then there are only three possibilities for $\tau$, and the manifold $M$ can only be one of the following examples. For the case where the manifold $M$ has dimension $n$, the following examples were introduced in \cite{FM}.
\begin{enumerate}
    \item Suppose $\Sigma=\Sigma^0\times\{-1, 1\}$. Represent a point $p \in M\cong G\times_Q\Sigma$ by $[(k,x)]$ with $k \in K$ and $x \in \Sigma^0$. The assignment $p \mapsto (kC^0, x) \in K/C^0\times\Sigma^0$ is well defined because the stabilizer of $\Sigma^0$ intersect $K$ equals $C^0$ in this case. Then, this map is a real analytic diffeomorphism $M \rightarrow \Snone\times\Sigma^0$ with a faithful, fiber-preserving $G$-action.

    \item If $\Sigma=\Sigma^0$ and $\tau$ is trivial, then $M$ has a well-defined real analytic diffeomorphism to $K/C\times\Sigma^0 \cong \RP\times\Sigma^0$, with a fiber-preserving action factoring through $\PSL$.

    \item If $\Sigma=\Sigma^0$ and $\tau$ acts freely. In this case, the stabilizer in $K$ of $\Sigma^0$ is $C$; note that $C=\langle\sigma, C^0\rangle$ and $\sigma$ normalizes $C^0$. Given $p \in M$ corresponding to $[(k,x)]$ with $k \in K$ and $x \in \Sigma^0$, there is a well defined map to the orbit $\{(kC^0, x), (k\sigma C^0, \tau.x)\} \in (K/C^0 \times \Sigma^0)/\langle\sigma\rangle \cong (\Snone\times\Sigma^0)/\langle\sigma\rangle$. This case, $M$ is a flat bundle with $Z_2$ monodromy over $\RP$, with a faithful $G$-action.
\end{enumerate}

In cases where the involution $\tau$ has fixed points on $\Sigma^0$, the dimension of the set of $\tau$-fixed points may vary, leading to multiple possible cases. If this fixed points set is $0$-dimensional, the manifold can be constructed by a blow-up at those points, as discussed in \cite{FM}. However, there is no general formula for an arbitrary dimension of the $\tau$-fixed set. 

\begin{lem} \label{lem5.4.}
    Under the same setting and assumptions in Theorem \hyperref[thm1.1.]{1.1}, and assuming the manifold $M$ of dimension $2n-3$ contains orbits conjugate to $\Gr$, only case (2) of Theorem \hyperref[thm1.1.]{1.1} can occur.
\end{lem}

\begin{proof}
    Since $M$ contains $\Gr$-orbits by assumption, consider the fixed point set $\Sigma_2$ of the restricted $C^0_2$-action, as defined in Proposition \hyperref[prop5.2.]{5.2}. For each $x\in\Sigma_2$, the point stabilizer $G_x$ contains $C^0_2$. Because $M$ contains an orbit conjugate to $\mathrm{Gr}(2,n)$, the stabilizer is conjugate in $G$ into the maximal parabolic subgroup $Q_2$ which stabilizes a $2$-plane in $\R^n$. Let $g\in G$ be such that $gG_xg^{-1}\leq Q_2$, and hence $gC^0_2g^{-1}\leq Q_2$. Note that the homogeneous spaces $G/Q_2$ and $K/C_2$ are $K$-equivariantly diffeomorphic up to finite cover. Then we have $$C^0_2\leq\mathrm{Stab}_G(g^{-1}Q_2)\cap K=\mathrm{Stab_K(g'C_2)}$$ for some $g'\in K$, where the stabilizers are taken with respect to the action by translation. Since $C^0_2\subset\mathrm{Stab}_k(g'C_2)$, it follows $g'\in N_K(C^0_2)=C_2$, the normalizer of $C^0_2$ in $K$. Hence $g\in Q_2$, and therefore $G_x\leq Q_2$ for all $x\in\Sigma_2$.
    
    By the orbit classification results in Theorem \hyperref[thm3.1.]{3.1}, the point stabilizer $G_x$ of $x\in\Sigma_2$ is one of $Q_2$, $Q^0_2$, or codimension one subgroups of $Q_2$ which are described in Section \hyperref[s4]{4}. Since all these subgroups contain $C^0_2$, it follows $Q_2.\Sigma_2=\Sigma_2$. Moreover, all codimension one subgroups of $Q_2$ act trivially in restriction to $\Sigma_2$. It follows that the $Q_2$-action on $\Sigma_2$ factors through the epimorphism $$\nu_2:Q_2\rightarrow\R^*$$ introduced in Section \hyperref[s4]{4}. Let $\Sigma^0_2$ be a connected component of $\Sigma_2$. The $\R^*_{>0}$ preserves $\Sigma^0_2$; this action is an analytic flow $\{\psi^t_{2,X}\}$, which corresponds to the restriction of $\{a^t_2\}$. Let $\sigma_2$ be an involution that leaves $\Sigma_2$ invariant. Let $\tau_2=\sigma_2|_{\Sigma_2}$. Depending on $\tau_2$, we set $\Sigma_2=\Sigma^0_2$ or $\Sigma^0_2\times\{-1,1\}$. Let $\mu_{2,X,\tau_2}$ denote the corresponding $Q_2$-action on $\Sigma_2$. For this $Q_2$-action, we define the $G$-equivariant map $$\Phi_2:G\times_{Q_2}\Sigma_2\rightarrow M ~~~[(g,x)]\mapsto g.x.$$
    
    Remains to show $\Phi_2$ is a diffeomorphism. Firstly, we verify surjectivity by showing it is open and closed. The image of $\Phi_2$ is closed because the fiber product is compact. To show $\Phi_2$ is open, we consider the orbit $G.x$ for all $x\in\Sigma_2$. If $G.x$ is $(2n-3)$-dimensional, it is open by dimension comparison. If $G.x$ is $(2n-4)$-dimensional, corresponding to $G_x=Q_2$ and $G.x=\Gr$, and $C^0_2$-fixed set in $G.x$ is $0$-dimensional. Thus $G.x$ is transverse to $\Sigma_2$ at $x$, and the differential map has full rank, which means $\Phi_2$ is open, so a surjective local diffeomorphism. We can verify injectivity in the same way as in the preceding proof using the Iwasawa decomposition, thus the map is a diffeomorphism.  
\end{proof}

\subsection{$G$-action on manifolds containing global fixed points}

When the $G$-action on a manifold admits global fixed points, the $G$-fixed point set has dimension $m-n$, by the Proposition \hyperref[prop3.5.]{3.5}. Since the $G$-fixed point set is sufficiently large, orbit types of higher dimension, such as $\Gr$, cannot occur. The following theorem describes all possible $G$-actions on manifolds with global fixed points.

\thmB*

\begin{proof}
    Suppose that $x_0 \in M$ is a $G$-fixed point. Since the manifold contains a $G$-fixed point, it contains $(m-n)$-dimensional $G$-fixed submanifolds by Proposition \hyperref[prop3.5.]{3.5}. Let $F^0$ be the connected component of $G$-fixed set that contains $x_0$. There exists a connected component $\Sigma^0$ of $C^0$-fixed point set $\Sigma$ introduced in Proposition \hyperref[prop5.1.]{5.1}, which contains $F^0$ as a submanifold. Thus, we have $x_0 \in F^0 \subseteq \Sigma^0 \subseteq \Sigma$. 

    As the previous proof of Lemma \hyperref[lem5.3.]{5.3}, $\Sigma$ is the $Q$-invariant set, which means $Q.\Sigma=\Sigma$. With respect to the isotropy representation of $G$ at $x_0$, there exists a line $l_0$ in a $Q$-invariant space tangent to $\Sigma^0$, which is not fixed by $G$. This isotropy representation is the standard representation on $\R^m$ as stated in Remark \hyperref[rmk3.6.]{3.6}. Recall that $\{a^t\}$ is the 1-parameter subgroup in the center of the Levi group $L\cong\GL$ in the Levi decomposition of $Q$. In a suitable parametrization, $\rho(a^t)$ has the eigenvalue $e^t$ on $l_0$. Let $\sigma \in C$ be an involution so that $\rho(\sigma)$ acts as $-Id$ on $l_0$. Both $\sigma$ and $\{a^t\}$ have no fixed points on $\Sigma^0\backslash F^0$ in a neighborhood of $x_0$. By Theorem \hyperref[thm3.1.]{3.1}, the stabilizers of these points in this neighborhood are $E^0$ and the corresponding orbits are $\R^n\backslash\{0\}$. Then an $n$-dimensional $G$-orbit fills a neighborhood of $x_0$ on $M\backslash F^0$. Finally, there is a neighborhood $U_0$ in $M$ at $x_0$ that is $G$-equivariantly homeomorphic to $\RO\times V_\alpha$, where $V_\alpha$ is an open neighborhood of $x_0$ in $F$. By applying Theorem \hyperref[thm2.3.]{2.3}, we obtain that the $G$-action on $U_0$ is real analytically equivalent to the standard representation $\rho$. 

    For any $G$-fixed points $x_0\in F$, there exists a local $G$-equivariant isomorphism $\Psi_\alpha$ from $\RO\times V_\alpha$ onto an open set in $M$, where $V_\alpha$ is an open neighborhood of $x_0$ in $F$. Since $F$ is a finite disjoint union of closed submanifolds in $M$, we can construct a finite open covering $\{V_\alpha\}$ of $F$ with a corresponding local isomorphism $\Psi_\alpha$ for each $V_\alpha$. 
    
    For each connected component $F_0$ of $F$, we can construct an open set $U=\bigcup \Psi_\alpha (\RO\times V_\alpha)$ in $M$ for a finite covering $\{V_\alpha\}$ of $F_0$. This open set $U$ contains $F_0$ and is $G$-equivariantly homeomorphic to $\RO\times F^0$.  

    \medskip
    
    Let $E$ be the intersection of $U$ and $\Sigma^0\backslash F$. Then $E$ forms a tubular neighborhood of $F^0$ within $\Sigma^0\backslash F$, and it is diffeomorphic to the normal bundle of $F^0$ in $\Sigma^0$ due to the existence of a real-analytic projection from $E$ to $F^0$. The normal bundle of $F^0$ consists of vectors that are normal to $F^0$ within $\Sigma^0$. Since the codimension of $F^0$ in $\Sigma^0$ is 1, the normal bundle is a real line bundle, and $E$ has two connected components separated by the zero section $F^0$. 

    Consider the holonomy of this line bundle arising from parallel transport of normal vectors along loops in $F^0$. The holonomy group is isomorphic to $\Z_2$, since after parallel transport, a normal vector may either remain in the same component or switch to the other. The corresponding holonomy representation defines a homomorphism $\pi_1(F^0)$ to $\Z_2$, whose kernel is a subgroup of index 2. 
    
    Next, we define a left principal $Q^0$-bundle $p$ from $E$ to $F^0$, satisfying $p
    (\Psi_\alpha(t,x))=x$ for $t\in\RO$. If the holonomy of this bundle is trivial, in other words, the monodromy representation $\pi_1(F^0)\rightarrow Q^0$ is trivial, then the bundle $p$ is globally trivial. Otherwise, 
    consider the associated holonomy representation of the normal line bundle to $F^0$ in $\Sigma^0$, which takes values in $\Z_2$, due to the codimension being one. The kernel of this representation corresponds to a double cover of $F^0$, yielding a principal $\Z_2$-bundle. Pulling back the bundle $p$ to this double cover trivializes it. Thus, we obtain a global real-analytic 
    $$\Psi : \R^n\times F^0 \to M,$$ which is a $G$-equivariant diffeomorphism onto an open set of $M$ satisfying $\Psi(0,x)=x$ for any $x\in F$. The same procedure can be applied to any connected component $F^0$ of $F$, allowing us to define the $G$-equivariant isomorphism $\Psi$ from $\R^n\times F$.

    \medskip
    
    Now, let $\Sigma^0_\beta$ be a connected component of $\Sigma^0\backslash F$. The number of such components is finite since the $G$-fixed set $F$ has only finitely many connected components. Each $\Sigma^0_\beta$ contains at least one codimension-one submanifold as a boundary arising from $F$. Since $\Sigma^0_\beta$ is connected and contains no $G$-fixed points, it is $Q^0$-invariant. 
    By the orbit classification in Theorem \hyperref[thm3.1.]{3.1} and the same reasoning as in the proof of Lemma \hyperref[lem5.3.]{5.3}, the stabilizers of points in $\Sigma^0_\beta$ are equal to either $Q^0$ or $E^0$. We define a map $$\Phi_\beta : G\times_{Q^0}\Sigma^0_\beta \rightarrow M$$ by $\Phi_\beta([g,x]) \rightarrow g.x$ for each $\Sigma^0_\beta$. As shown in the proof in Lemma \hyperref[lem5.3.]{5.3}, $\Phi_\beta$ is a local diffeomorphism and has an open image in $M$, which is diffeomorphic to $\Snone \times \Sigma^0_\beta$ under the $K$-equivariant diffeomorphism. 

    For each boundary component $F^0$ of $\Sigma^0_\beta$, define $I_\beta$ as the intersection $U\cap\Sigma^0_\beta$, where $U$ is the open neighborhood around $F^0$ in $M$ defined above. The closure of $I_\beta$ can be considered as a collar neighborhood of $\Sigma^0_\alpha$ which has the same structure as $[0, 1) \times F^0$. In other words, there exists a diffeomorphism $f$ from $[0, 1)\times F^0$ to a collar neighborhood $\overline{I_\beta}$ satisfying $f(0,x)=x$ for all $x \in F^0$. 

    The restriction of $\Phi_\beta$ to $G\times_{Q^0} I_\beta$ is a $G$-equivariant diffeomorphism onto $U\backslash F$. Under the $K$-equivariant identification with $\Snone\times I_\beta$, the fibers $\{p\}\times I_\beta$ are diffeomorphic to $[0,1)\times F^0$. Here, $[0,1)$ corresponds to the rays from the origin in $\R^n$ within the neighborhood $U$. With this identification, we can define a map from $G\times_{Q^0} \Sigma^0_\beta$ to $\RO \times F$ for every $\Sigma^0_\beta$ near their boundaries. 
    
    Finally, the disjoint union of the image of $\Psi$ and the images of $\Phi_\beta$ for all $\beta$, identified by the above identification map, forms an open and closed submanifold in $M$ with real-analytic $G$-action. Consequently, it is equal to $M$, as described in Construction 2 where $\{\psi^t_X\}$ coincides with $\{a^t\}$ when restricted to each $\Sigma^0_\alpha$.
\end{proof}

\section{Density/Non-density of structural stability}\label{s6}

Let $\X$ be the space of vector fields on a compact differentiable manifold $\Sigma$, equipped with the $C^r$-topology, $r\geq1$. A vector field $X\in\X$ is said to be $structurally$ $stable$ if there exists a neighborhood $U_X$ of $X$ in $\X$ such that every $Y\in U_X$ is topologically equivalent to $X$. Two vector fields $X$ and $Y$ in $\X$ are $ topologically$ $equivalent$ if there exists a homeomorphism $H:\Sigma \rightarrow\Sigma$ that maps the trajectories of $X$ onto those of $Y$. In other words, a vector field is structurally stable if small perturbations in the $C^r$-topology do not change the topological structure of its solution curves. 

\medskip

Peixoto in \cite{P} provided the density  structural stability on $2$-dimensional manifolds. Structural stability means all finitely many fixed points are hyperbolic. In dimension one, the manifold $\Sigma$ is diffeomorphic to a circle $\Sone$ or a finite union of circles. This case is treated in \cite{P} as a trivial special case. 

\begin{thm} [{\cite{P}}]\label{thm6.1.}
    The set of all structurally stable vector fields is open and dense in $\X$, the space of $C^1$-vector fields with the $C^1$-topology on a compact manifold $\Sigma$ of dimension $n\leq2$.
\end{thm}

Smale in \cite{Smale} gave a negative answer to the question of the density of structurally stable systems in dimension $n=4$, and this result extends straightforwardly to all higher dimensions $n\geq4$. In dimension $n=3$, the non-density of structurally stable systems was provided by Williams in \cite{W}.

\begin{thm} [{\cite{Smale},\cite[Theorem B]{W}}]\label{thm6.2.}
    For $n\geq3$, there exists a compact differentiable $n$-dimensional manifold $\Sigma$ and an open $U\in\X$ in the space of $C^r$-vector fields with the $C^r$-topology on $\Sigma$, $r\geq1$, such that there is no $X\in U$ structurally stable.
\end{thm}

The next theorem implies the non-density of structurally stable systems in the space of $C^r$-diffeomorphisms, $r\geq2$, on a compact $n$-manifold with $n\geq2$.

\begin{thm}[Newhouse phenomenon {\cite{N}}] \label{thm6.3.} For any compact manifold of dimension greater than one, there exist open sets in the $C^r$-topology, for $r\geq2$, of non-hyperbolic diffeomorphisms. 
\end{thm}

Since structurally stable systems satisfy Axiom A (which requires that all nonwandering points have hyperbolic behavior) and the Newhouse phenomenon does not meet the Axiom A condition, they cannot be structurally stable.

\medskip

By Theorem \hyperref[thm1.1.]{1.1}, any smooth $G$-action on a manifold $M$ without fixed points is determined by a smooth vector field $X$ on $\Sigma$ or $\Sigma_2$ and an involution $\tau$ of $\Sigma$ or $\Sigma_2$ commuting with $X$. Consequently, the density of structurally stable vector fields on $\Sigma$ or $\Sigma_2$ implies the density of structurally stable $G$-actions on the manifold $M$. 

\begin{proof}[Proof of Corollary \ref{cor1.3.}]
    By Theorem \hyperref[thm1.1.]{1.1}, the manifold $M$ is obtained as a suspension space of actions of maximal parabolic subgroups on a submanifold $\Sigma$ or $\Sigma_2$ except for the transitive action cases which cannot be deformed. This action factors through an $\R^*$-action on $\Sigma$ or $\Sigma_2$ generated by flow, which means that a smooth vector field on $\Sigma$ or $\Sigma_2$ determines a smooth $G$-action on the manifold $M$. If the compact manifold $\Sigma$ has dimension $n\leq2$, then the set of structurally stable $C^1$-vector fields is dense by Theorem \hyperref[thm6.1.]{6.1}. Since $\Sigma_2$ is a finite union of circles, case (2) of Theorem \hyperref[thm1.1.]{1.1} also has a density of structural stability. On the other hand, if $\mathrm{dim}(\Sigma)\geq2$, then the space of structurally stable $C^r$-diffeomorphisms, $r\geq2$, is not dense by Theorem \hyperref[thm6.3.]{6.3}. Moreover, if $\mathrm{dim}(\Sigma)>2$, then structurally stable vector fields are not dense by Theorem \hyperref[thm6.2.]{6.2}.
\end{proof}

\bibliographystyle{amsplain}

\end{document}